\documentclass[12pt]{amsart}
\usepackage[english]{babel}
\usepackage{amssymb, amsmath, amsthm}
\usepackage[utf8]{inputenc} 
\usepackage[T1]{fontenc}
\usepackage[shortlabels]{enumitem}
\usepackage{hyperref}
\usepackage[top=1in, bottom=1in, left=1.25in, right=1.25in]{geometry}

\usepackage{graphicx}
\graphicspath{ {./images/} }
\usepackage{subcaption}
\usepackage{float}

\newlist{symbols}{itemize}{1}
\setlist[symbols,1]{label=,labelwidth=1.5in,align=parleft,itemsep=0.1\baselineskip,leftmargin=!}

\usepackage{epsfig}
\usepackage{url}
\usepackage{setspace}
\usepackage{color}
\usepackage{mathrsfs}
\usepackage{amsmath}
\usepackage{amsthm}
\usepackage{stmaryrd}

\usepackage{tikz}
\usepackage{pgf}
\usetikzlibrary{arrows}
\usetikzlibrary{automata,positioning}

\newtheorem{theorem}{Theorem}[section]
\newtheorem{cor}[theorem]{Corollary}
\newtheorem{lem}[theorem]{Lemma}
\newtheorem{prop}[theorem]{Proposition}
\newtheorem{ques}[theorem]{Question}
\newtheorem{conj}[theorem]{Conjecture}

\newtheorem{exam}[theorem]{Example}
\newtheorem{obs}[theorem]{Observation}

\theoremstyle{definition}
\theoremstyle{remark}
\newtheorem{rem}[theorem]{Remark}

\renewcommand{\phi}{\varphi}

\newcommand{\ord}{{\rm ord}}

\newcommand{\Mod}[1]{\ \mathrm{mod}\ #1}

\allowdisplaybreaks

\title{$M$-ary partition polynomials}

\author{B\l{}a\.zej \.Zmija}
\address{Institute of Mathematics of the Polish Academy of Sciences, ul. \'{S}niadeckich 8, 00-656 Warszawa, Poland} 
\email{blazej.zmija@gmail.com}
\keywords{$M$-ary partitions, congruences, generating functions}
\subjclass[2020]{11P83, 11P81, 05A15, 05A17}
\date{\today}

\begin{document}

\begin{abstract}
Let $M=(m_{i})_{i=0}^{\infty}$ be a sequence of integers such that $m_{0}=1$ and $m_{i}\geq 2$ for $i\geq 1$. In this paper we study $M$-ary partition polynomials $(p_{M}(n,t))_{n=0}^{\infty}$ defined as the coefficient in the following power series expansion:
\begin{align*}
\prod_{i=0}^{\infty}\frac{1}{1-tq^{M_{i}}} = \sum_{n=0}^{\infty} p_{M}(n,t)q^{n},
\end{align*}
where $M_{i}=\prod_{j=0}^{i}m_{j}$. In particular, we provide a detailed description of their rational roots and show, that all their complex roots have absolute values not greater than $2$. We also study arithmetic properties of $M$-ary partition polynomials. One of our main results says that if $n=a_{0}+a_{1}M_{1}+\cdots +a_{k}M_{k}$ is a (unique) representation such that $a_{j}\in\{0,\ldots ,m_{j+1}-1\}$ for every $j$, then
\begin{align*}
p_{M}(n,t)\equiv t^{a_{0}}\prod t^{a_{j}}f(a_{j}+1,t^{m_{j}-1}) \pmod{g_{k}(t)},
\end{align*}
where $f(a,t):=\frac{t^{a}-1}{t-1}$ and $g_{k}(t):=\gcd \big(t^{m_{1}+m_{2}-1}f(m_{2},t^{m_{1}-1}),\ldots ,t^{m_{k}+m_{k+1}-1}f(m_{k+1},t^{m_{k}-1})\big)$. This is a polynomial generalisation of the well-known characterisation modulo $m$ of the sequence of $m$-ary partition.
\end{abstract}

\maketitle

\section{Introduction}

If $m$ is a positive integer, then by an $m$-ary partition of $n$ we understand any expression of the form
\begin{align*}
n=m^{i_{1}} + m^{i_{2}} + \cdots + m^{i_{k}},
\end{align*}
where $i_{1}\leq i_{2}\leq \ldots \leq i_{k}$ are nonnegative integers. We will denote the number of such partitions of  $n$ by $b_{m}(n)$.

The first person to study $m$-ary partitions was probably Euler, who found the generating function for the sequence of binary partitions, that is, the case of $m=2$. Later, the sequence $(b_{m}(n))_{n=0}^{\infty}$ was studied by many mathematicians, and for example Mahler \cite{Mah}, and then Pennington \cite{Penn}, described its the asymptotic behaviour, Churchhouse started studying its arithmetic properties in the case of $m=2$ in \cite{Ch}, and this was continued for more general $m$ by Gupta \cite{Gup}, Rødseth \cite{R}, Rødseth and Sellers \cite{RS}, and many more.

The notion of $m$-ary partitions was generalised by many authors and many directions, see for example \cite{ABRS,And2,CS,GS,Hou,LuM,MaLu,BZ4}. For us, the two types of generalisations will be important. The first one, done by Folsom et al. in \cite{Fol} is the following. Let $M=(m_{n})_{n=0}^{\infty}$ be a fixed sequence of natural numbers such that $m_{0}=1$ and $m_{j}\geq 2$ for $j\geq 1$. For every $i\geq 0$ let $M_{i}:=\prod_{j=0}^{i}m_{j}$. Then we will call every representation of the form
\begin{align*}
n=M_{i_{1}}+M_{i_{2}}+\cdots +M_{i_{k}}
\end{align*}
an $M$-ary partition of $n$. In the paper, we will denote by $b_{M}(n)$ the number of $M$-ary partitions of $n$.

Note, that Folsom et al. called such representations $M$-sequence non-squashing partitions and were focused on their asymptotic behaviour. They also stated some conjectures concerning the divisibility of numbers $b_{M}(n)$ by the terms in the sequence $M$. These conjectures were however not true as was showed in \cite{BZCong}.

On the other way Ulas and Żmija introduced in \cite{BZ2} a polynomial generalisation of binary partitions, that is, binary partition polynomials, defined as the coefficients in the following power series expansion:
\begin{align*}
F_{2}(q,t):=\prod_{j=0}^{\infty}\frac{1}{1-tq^{2^{j}}}=\sum_{n=0}^{\infty}p_{2}(n,t)q^{n}.
\end{align*}
They used the polynomials to obtain some non-trivial identities involving binary partitions. For example, they proved that for every $n$, the number of binary partitions of $8n+4$ with the number of parts $\equiv 2k \pmod{4}$ is equal to the number of binary partitions of $8n+4$ with the number of parts $\equiv 2k+1 \pmod{4}$  for $k\in\{0,1\}$. They also used binary partition polynomials to study the sequence of so-called $s$-partitions. These are partitions into parts of the form $2^{k}-1$, $k\in\mathbb{N}_{0}$. Note here, that $s$-partitions were introduced by Bhatt in \cite{Bhatt}, who showed their application in cryptography.

In this paper, we study properties of (probably) the simplest non-trivial case of polynomials related to partitions, that is, $M$-ary partition polynomials. Here we are interested exclusively in the case of partitions with parts from an infinite set. This corresponds to the case of $M$ being infinite. However, many of the results from this chapter remains true also if $M$ is finite.

Let $M=(m_{n})_{n=0}^{\infty}$ be a fixed sequence of natural numbers such that $m_{0}=1$ and $m_{j}\geq 2$ for $j\geq 1$. Let $M_{n}:=\prod_{i=0}^{n}m_{i}$. For a number $n$ we define the $n$th $M$-ary partition polynomial $p_{M}(n,t)$ as the $n$th coefficient in the following power series expansion:
\begin{align*}
F_{M}(q,t):=\prod_{j=0}^{\infty}\frac{1}{1-tq^{M_{j}}}=\sum_{n=0}^{\infty}p_{M}(n,t)q^{n}.
\end{align*}
The coefficients of the polynomial $p_{M}(n,t)$ have the following combinatorial interpretation. If
\begin{align*}
p_{M}(n,t)=\sum_{j=0}^{n}a_{M}(j,n)t^{j},
\end{align*}
then $a_{M}(j,n)$ is equal to the number of representations of $n$ as a sum of the numbers $M_{i}$ with exactly $j$ parts. In particular, $p_{M}(n,1)=b_{M}(n)$. Note, that the case of binary partition polynomials is obtained by putting $M=(1,2,2,\ldots )$. % Every representation of the form $n=M_{i_{1}}+M_{i_{2}}+\cdots +M_{i_{k}}$ will be called an $M$-ary partition of $n$. The fact that $\deg p_{M}(n,t)=n$ will be proved in Observation \ref{ObsDegree} below.

Let us describe briefly the content of this paper. In the next section we will describe the basic properties of $M$-ary partition polynomials and show some examples. Then we will focus on their rational roots in Section \ref{SectionRationalRoots} and complex roots in Section \ref{SectionComplexRoots}. In these part we will mainly interested in the case of $m$-ary partition polynomials, that is, the case of $M=(1,m,m,\ldots )$. In particular, we will obtain a very precise description of all the possible rational roots and that all the complex roots are in the circle of radius $2$ and the center in the origin. In Section \ref{SectionCertainPropertiesModulo} we will return to the general case and prove a polynomial generalisation of the well-known characterisation of the sequence $(b_{m}(n))_{n=0}^{\infty}$ modulo $m$. Moreover, our result works for any sequence $M$, not only $M=(1,m,m,\ldots )$. In the last section we will study two sequences of the coefficients of $M$-ary partition polynomials. In particular, we will show that both these sequences are bounded, and connect the values of one of them with the numbers of partitions into parts of the form $M_{i}-1$.

In the paper, we will sometimes use the following notion of automaticity. Let $k\in\mathbb{N}$ and $k\geq 2$. We say that a sequence $\mathbf{a}=(a_{n})_{n=0}^{\infty}$ is $k$-automatic if the following set
\begin{align*}
\mathcal{K}_{k}(\mathbf{a}):=\big\{\ (a_{k^{i}n+j})_{n=0}^{\infty} \ \big|\ i\in\mathbb{N}_{0} \textrm{ and } 0\leq j<k^{i} \ \big\},
\end{align*}
called the $k$-kernel of $\mathbf{a}$, is finite. This condition is equivalent to the condition that there exists a $k$-automaton for the sequence $\mathbf{a}$. Roughly speaking, $k$-automaton (or automaton for short if $k$ is fixed) is a procedure that allows us, for a given $n$, find the value of $a_{n}$ using only the representation of $n$ in base $k$. For more precise definition see \cite{AS}.

One of the most important results in the theory of automatic sequences is the following characterisation in the case of $k$ being a power of a prime.

\begin{lem}[Christol's Theorem]\label{ThmChristol}
Let $q$ be a power of a prime. Let $\mathbf{a}=(a_{n})_{n=0}^{\infty}$ be a sequence over $\mathbb{F}_{q}$. Then $\mathbf{a}$ is $q$-automatic if and only if its generating function is algebraic over $\mathbb{F}_{q}(x)$.
\end{lem}
\begin{proof}
Can be found in \cite{AS}.
\end{proof}

\section{Basic properties}

In this section we collect the most basic properties of the $M$-ary partition polynomials. From the definition of the function $F_{M}(q,t)$ we have
\begin{align*}
F_{M}(q,t)=\frac{1}{1-tq}F_{M'}(q^{m_{1}},t),
\end{align*}
where $M'=(1,m_{2},m_{3},\ldots )$. Therefore,
\begin{align*}
\sum_{n=0}^{\infty}p_{M'}(n,t)q^{m_{1}n}=(1-tq)\left(\sum_{n=0}^{\infty}p_{M}(n,t)q^{n}\right)=\sum_{n=0}^{\infty}\big(p_{M}(n,t)-tp_{M}(n-1,t)\big)q^{n}.
\end{align*}
By comparing the coefficients on both sides of the above power series we get
\begin{align}\label{RecurRelationsMaryPartitions}
\left\{\begin{array}{ll}
p_{M}(m_{1}n+j,t)=tp_{M}(m_{1}n+j-1,t), & \textrm{ if } j\in \{1,\ldots ,m_{1}-1\}, \\
p_{M}(m_{1}n,t)=tp_{M}(m_{1}n-1,t)+p_{M'}(n,t). & 
\end{array}\right.
\end{align}
In particular, we can compute the first few $M$-ary partition polynomials.
\begin{exam}\label{ExampleSmallCases}
If $n=km_{1}+j$ for some $j\in\{0,\ldots ,m_{1}-1\}$ and $k\in\{0,\ldots ,m_{2}-1\}$, then
\begin{align*}
p_{M}(km_{1}+j,t) = t^{k+j} \frac{t^{(m_{1}-1)(k+1)}-1}{ t^{m_{1}-1} -1}.
\end{align*} 
\end{exam}
\begin{proof}
Because of the first equality in \eqref{RecurRelationsMaryPartitions}, it is enough to justify the formula in the case of $j=0$. If $k=0$ it is trivial. Assume it works for some $k-1\in\{0,\ldots ,m_{2}-2\}$. We get for $k$:
\begin{align*}
p_{M}(km_{1},t) & =t^{m_{1}}p_{M}( (k-1)m_{1},t) + p_{M'}(k,t) = t^{m_{1}}t^{k-1} \frac{t^{(m_{1}-1)k}-1}{ t^{m_{1}-1} -1} + t^{k} \\
 & = t^{k} \frac{t^{m_{1}-1} t^{(m_{1}-1)k} -t^{m_{1}-1} + t^{m_{1}-1} -1}{t^{m_{1}-1} - 1} = t^{k} \frac{t^{(m_{1}-1)(k+1)}-1}{ t^{m_{1}-1} -1}.
\end{align*}
In the above computations we used the fact that $p_{M'}(k,t) = t^{k}$ for $k\in\{0,\ldots ,m_{2}-1\}$ which follows trivially from the first equation in \eqref{RecurRelationsMaryPartitions} used with the sequence $M'$ instead of $M$.
\end{proof}

Relations \eqref{RecurRelationsMaryPartitions} allow us to find the degree of the polynomial $p_{M}(n,t)$. The formula in theorem below directly generalizes the result from \cite{BZ2} concerning the case $m=2$.

\begin{obs}\label{ObsDegree}
For every $M$ and $n$ we have $\deg p_{M}(n,t)=n$.
\end{obs}
\begin{proof}
We have $p_{M}(0,t)=1$ and $p_{M}(1,t)=t$. Assume that $\deg p_{M}(k,t)=k$ for $k<n$. Thus if $m_{1}\nmid n$, then by the induction hypothesis and relations \eqref{RecurRelationsMaryPartitions} we obtain
\begin{align*}
\deg p_{M}(n,t)=\deg \left( tp_{M}(n-1,t)\right)=1+\deg p_{M}(n-1,t)=1+(n-1)=n.
\end{align*}
If $m_{1}\mid n$, then by the induction hypothesis we get
\begin{align*}
\deg \left(t p_{M}(n-1,t)\right)=1+\deg p_{M}(n-1,t)= n>\frac{n}{m_{1}}=\deg p_{M'}\left(\frac{n}{m_{1}},t\right).
\end{align*}
We finish the proof analogously as in the previous case.
\end{proof}

\section{Rational roots}\label{SectionRationalRoots}

The first objects that we want to study in the context of polynomials in general, and the $M$-ary polynomials in particular, are their roots. For now, we focus on the existence of rational roots. This will naturally lead us to studying the sequence $(p_{M}(n,-1))_{n=0}^{\infty}$. Note that $p_{M}(n,-1)$ is equal to the difference between the number of $M$-ary partitions with even number of parts and the number of $M$-ary partitions with odd number of parts. 

In this section we narrow our considerations down to the case of $m$-ary partition polynomials where $m\in\mathbb{N}$. This is a special case of $M$-ary partition polynomials with $M=(1,m,m,m,\ldots )$. We will write $p_{m}(n,t)$ instead of $p_{M}(n,t)$ in this case.

In this case we simply have $M'=M$, so the recurrence relations \eqref{RecurRelationsMaryPartitions} simplify to
\begin{align}\label{RecurRelationsmaryPartitions}
\left\{\begin{array}{ll}
p_{m}(0,t)=1 , &  \\
p_{m}(mn+j,t)=tp_{m}(mn+j-1,t), & \textrm{ if } j\in \{1,\ldots ,m-1\}, \\
p_{m}(mn,t)=tp_{m}(mn-1,t)+p_{m}(n,t). & 
\end{array}\right.
\end{align}

Although the proof of the next lemma uses Theorem \ref{ThmRootsPm}, we state the lemma first to give the reader an insight about possible values of rational roots of $p_{m}(n,t)$.

\begin{lem}\label{LemRationalRoots}
If $r$ is a rational root of $p_{m}(n,t)$ for some $n$ then $r\in\{-1,0\}$.
\end{lem}
\begin{proof}
A simple induction argument (together with Theorem \ref{ThmRootsPm} below) shows that for every $n$, the leading coefficient and the initial coefficient in $p_{m}(n,t)$ are both equal to $1$. Hence, $r$ has to be an integer and divide $1$ or be equal to $0$. Moreover, all nonzero coefficients in $p_{m}(n,t)$ are positive, so $p_{m}(n,1)>0$. Therefore, $r\in\{-1,0\}$.
\end{proof}

In order to state the next result we need the following notation. If $n=\sum_{j=0}^{k}n_{j}m^{j}$ is a unique $m$-ary representation of $n$, then we define
\begin{align*}
s_{m}(n):=\sum_{j=0}^{k}n_{j}.
\end{align*}
That is, $s_{m}(n)$ is the sum of digits in the $m$-ary representation of $n$. 

The next theorem provides information about the order of $0$ as a root of the polynomial $p_{m}(n,t)$.

\begin{theorem}\label{ThmRootsPm}
For every natural number $n$ the following formula holds:
\begin{align*}
\ord_{t=0}p_{m}(n,t)=s_{m}(n).
\end{align*}
\end{theorem}
\begin{proof}
We proceed by induction. The statement is true for $n=0$ and $n=1$. Assume that it is true for some $n=mn'+a$, where $a\in\{0,\ldots ,m-2\}$. By the second equation in \eqref{RecurRelationsmaryPartitions} we have
\begin{align*}
\ord_{t=0}p_{m}(mn'+a+1,t)=1+\ord_{t=0}p_{m}(mn'+a,t)=1+s_{m}(mn'+a)=s_{m}(mn'+a+1).
\end{align*}
Hence, we can assume that $n=mn'+m-1$. Observe that $s_{m}(mn'-1)\geq s_{m}(n')$. Indeed, if $n'=\sum_{j=h}^{k}a_{j}m^{j}$, where $a_{h}$ is the smallest non-zero digit in the representation, then
\begin{align*}
mn'-1=\sum_{j=h+1}^{k}a_{j}m^{j+1}+ (a_{h}-1)m^{h+1}+\sum_{j=0}^{h}(m-1)m^{j}.
\end{align*}
Therefore,
\begin{align*}
s_{m}(mn'-1)=\sum_{j=h}^{k}a_{j}+(h+1)(m-1)-1\geq \sum_{j=h}^{k}a_{j}=s_{m}(n').
\end{align*}
Hence,
\begin{align*}
\ord_{t=0}p_{m}(mn'+m,t)= & \ord_{t=0}\big(tp_{m}(mn'+m-1,t)+p_{m}(n'+1,t)\big) \\
= & \min\big\{1+s_{m}(mn'+m-1),s_{m}(n'+1)\big\} \\ 
= & s_{m}(n'+1)=s_{m}(mn'+m),
\end{align*}
and this equality finishes the proof.
\end{proof}

Let us move to the problem when $-1$ is a root of some $m$-ary partition polynomial. Interestingly, the behaviour of polynomials $p_{m}(n,t)$ at $-1$ highly depends on the parity of $m$. We begin by studying the case of even $m$ in which we can provide a very good description of the numbers $p_{m}(n,-1)$.

\begin{theorem}\label{Thmmarypoly-1}
Let $n=a_{0}+a_{1}m+\cdots +a_{k}m^{k}$ be the base $m$ representation of $n$. If $2\mid m$, then
\begin{align*}
p_{m}(n,-1)=(-1)^{a_{0}}\prod_{j=1}^{k}\frac{1+(-1)^{a_{j}}}{2}.
\end{align*}
In particular $p_{m}(n,-1)\in\{-1,0,1\}$ and
\begin{align*}
p_{m}(n,-1)=0 \hspace{1cm} \textrm{ if and only if } \hspace{1cm} 2\nmid a_{j} \textrm{ for some } j\in\{1,\ldots ,k\}.
\end{align*}
\end{theorem}
\begin{proof}
Let $p_{n}:=p_{m}(n,-1)$. Relations \eqref{RecurRelationsmaryPartitions} imply
\begin{align*}
\left\{\begin{array}{ll}
p_{0}=1, & \\
p_{mn+j}=-p_{mn+j-1}, & \textrm{if } j\in\{1,\ldots , m-1\}, \\
p_{mn}=-p_{mn-1}+p_{m}.
\end{array}\right.
\end{align*}
In particular, we have $p_{mn+j}=(-1)^{j}p_{mn}$ for $j\in\{0,\ldots ,m-1\}$, and hence
\begin{align*}
p_{mn}=-p_{mn-1}+p_{n}=(-1)^{m}p_{m(n-1)}+p_{n}=p_{m(n-1)}+p_{n},
\end{align*}
because $m$ is even by the assumption. Therefore, if we  write $n_{1}=mn_{2}+a_{1}$, then
\begin{align*}
p_{mn_{1}}= &\ p_{m(n_{1}-1)}+p_{n_{1}}=p_{m(n_{1}-2)}+p_{n_{1}-1}+p_{n_{1}}=\cdots =\sum_{j=0}^{n_{1}}p_{j} \\
= &\ \sum_{l=0}^{n_{2}-1}\sum_{j=0}^{m-1}p_{ml+j} + \sum_{j=0}^{a_{1}}p_{mn_{2}+j}=\sum_{l=0}^{n_{2}-1}\sum_{j=0}^{m-1}(-1)^{j}p_{ml} + \sum_{j=0}^{a_{1}}(-1)^{j}p_{mn_{2}} \\
= &\ \sum_{l=0}^{n_{2}-1}\left(\sum_{j=0}^{m-1}(-1)^{j}\right)p_{ml} + \left(\sum_{j=0}^{a_{1}}(-1)^{j}\right)p_{mn_{2}} = 0+\frac{1+(-1)^{a_{1}}}{2}p_{mn_{2}}=\frac{1+(-1)^{a_{1}}}{2}p_{mn_{2}}.
\end{align*}

We now use the induction argument on the number of digits in the $m$-ary representation of $n$. If $n=a_{0}$ for some $a_{0}\in\{0,\ldots ,m-1\}$ then the result is true. Assume this is true for every natural number with at most $k$ digits in the $m$-ary representation. Let $n=a_{0}+a_{1}m+\cdots +a_{k}m^{k}$ (that is, $n$ has $k+1$ digits in its base $m$ representation). By our previous computations and the induction hypothesis we get
\begin{align*}
p_{mn_{1}}=\frac{1+(-1)^{a_{1}}}{2}p_{mn_{2}}=\frac{1+(-1)^{a_{1}}}{2}\prod_{j=2}^{k}\frac{1+(-1)^{a_{j}}}{2}=\prod_{j=1}^{k}\frac{1+(-1)^{a_{j}}}{2},
\end{align*}
and hence,
\begin{align*}
p_{n}=(-1)^{a_{0}}p_{mn_{1}}=(-1)^{a_{0}}\prod_{j=1}^{k}\frac{1+(-1)^{a_{j}}}{2},
\end{align*}
as desired.
\end{proof}

\begin{cor}\label{CorDiffEvenOdd}
Let
\begin{align*}
e_{m}(n):= &\ \#\left\{m-\textrm{ary partitions with even number of parts}\right\}, \\
o_{m}(n):= &\ \#\left\{m-\textrm{ary partitions with odd number of parts}\right\}.
\end{align*}
If $2\mid m$ and $n=a_{0}+a_{1}m+\cdots +a_{k}m^{k}$ is the base $m$ representation of $n$, then
\begin{align*}
e_{m}(n)-o_{m}(n)=(-1)^{a_{0}}\prod_{j=1}^{k}\frac{1+(-1)^{a_{j}}}{2}.
\end{align*}
In particular, $|e_{m}(n)-o_{m}(n)|\leq 1$.
\end{cor}
\begin{proof}
The statement follows from the observation that $p_{m}(n,-1)=e_{m}(n)-o_{m}(n)$, and Theorem \ref{Thmmarypoly-1}.
\end{proof}

%In order to state the next corollary, we need to introduce the notion of automaticity.

Now we show the automaticity of the sequence $(e_{m}(n)-o_{m}(n))_{n=0}^{\infty}$ if $m$ is even. The required definitions have been resented in the introduction.

\begin{cor}
If $2\mid m$, then the sequence $(e_{m}(n)-o_{m}(n))_{n=0}^{\infty}$ is $m$-automatic.
\end{cor}
\begin{proof}
We will prove that the $m$-kernel of the considered sequence is finite. Recall that $e_{m}(n)-o_{m}(n)=p_{m}(n,-1)=:p_{n}$ for all $n\in\mathbb{N}_{0}$. Theorem \ref{Thmmarypoly-1} implies the following relations for all $n\in\mathbb{N}_{0}$ and $0\leq a,b\leq m-1$:
\begin{align*}
p_{m(mn+a)+b}=(-1)^{b}\frac{1+(-1)^{a}}{2}p_{mn}.
\end{align*}
Hence,
\begin{align*}
p_{m^{2}n+b}= &\ (-1)^{b}p_{m^{2}n}=(-1)^{b}p_{mn}=\left\{\begin{array}{ll}
p_{mn}, & 2\mid b, \\
-p_{mn}, & 2\nmid b
\end{array}\right., \\
p_{m^{2}n+ma+b}= &\ \left\{\begin{array}{ll}
p_{m^{2}n+b}, & 2\mid a, \\
0, & 2\nmid a
\end{array}\right.
\end{align*}
In particular, every sequence of the form $(p_{m^{j}n+k})_{n=0}^{\infty}$ is equal to one of the sequences: $(p_{n})_{n=0}^{\infty}$, $(p_{mn})_{n=0}^{\infty}$, $(-p_{mn})_{n=0}^{\infty}$ or the constant sequence $(0)_{n=0}^{\infty}$. Thus
\begin{align*}
\mathcal{K}_{m}\big((p_{n})_{n=0}^{\infty}\big)=\big\{\ (p_{n})_{n=0}^{\infty},\ (p_{mn})_{n=0}^{\infty},\ (-p_{mn})_{n=0}^{\infty},\ (0)_{n=0}^{\infty}\ \big\}.
\end{align*}
is finite and the result follows.
\end{proof}

One can see that the sequence $(e_{m}(n)-o_{m}(n))_{n=0}^{\infty}$ is generated by the following automaton (reading the input starts with the least significant digit):
 \begin{center}
	\begin{tikzpicture}[->, shorten >= 1pt, node distance=2.5 cm, on grid, auto]
	\node[state] (c_0) {$1$};
	\node[inner sep=1pt] (c_s) [left=1cm of c_0] {};
	\node[state, inner sep=1pt] (c_1) [above right=1cm and 2.5cm of c_0] {$1$};
	\node[state, inner sep=1pt] (c_2) [below right=1cm and 2.5cm of c_0] {$-1$};
	\node[state, inner sep=1pt] (c_3) [below right=1cm and 2.5cm of c_1] {$0$};
	\path[->]
	(c_0) edge node {even} (c_1)
	edge node {odd} (c_2)
	(c_s) edge node {} (c_0)
	(c_1) edge [loop above] node {even} (c_1)
	edge node {odd} (c_3)
	(c_2) edge [loop below] node {even} (c_2)
	edge node {odd} (c_3)
	(c_3) edge [loop right] node {all} (c_3);
	\end{tikzpicture}
    \end{center}
where the term 'even' means any even digit in the base $m$ representation, and similarly 'odd' means any odd digit.

One can also give an automaton generating the sequence $(e_{m}(n)-o_{m}(n))_{n=0}^{\infty}$ where reading the input starts with the most significant digit. It is the following:
\begin{center}
	\begin{tikzpicture}[->, shorten >= 1pt, node distance=2.5 cm, on grid, auto]
	\node[state] (c_0) {$1$};
	\node[inner sep=1pt] (c_s) [left=1cm of c_0] {};
	\node[state, inner sep=1pt] (c_1) [right=2.5cm of c_0] {$-1$};
	\node[state, inner sep=1pt] (c_2) [right=2.5cm of c_1] {$0$};
	\path[->]
	(c_0) edge [loop above] node {even} (c_0)
	edge node {odd} (c_1)
	(c_s) edge node {} (c_0)
	(c_1) edge node {all} (c_2);
	\end{tikzpicture}
    \end{center}
We use the same convention using the terms 'even' and 'odd' as before.

The case of odd $m$ is completely different.

\begin{theorem}
If $2\nmid m$, then for every $n$:
\begin{align*}
p_{m}(n,-1)=(-1)^{n}p_{m}(n,1).
\end{align*}
In particular, $|p_{m}(n,-1)|\to\infty$ as $n\to\infty$, and $p_{m}(n,-1)\neq 0$ for all $n$'s.
\end{theorem}
\begin{proof}
Let
\begin{align*}
F_{m}(q,t)=\sum_{n=0}^{\infty}p_{m}(n,t)q^{n}=\prod_{j=0}^{\infty}\frac{1}{1-tq^{m^{j}}}.
\end{align*}
Then
\begin{align*}
\sum_{n=0}^{\infty}p_{m}(n,-1)q^{n}= &\ F_{m}(q,-1)=\prod_{j=0}^{\infty}\frac{1}{1+q^{m^{j}}}=\prod_{j=0}^{\infty}\frac{1}{1-(-q)^{m^{j}}} \\
= &\ F_{m}(-q,1)=\sum_{n=0}^{\infty}p_{m}(n,1)(-q)^{n}=\sum_{n=0}^{\infty}(-1)^{n}p_{m}(n,1)q^{n},
\end{align*}
because $m$ is odd. The result follows.
\end{proof}

The above theorem implies that if $m$ is odd, then the sequence $\big(e_{n}(n)-o_{m}(n)\big)_{n=0}^{\infty}$ is unbounded and
\begin{align*}
(-1)^{n}\big(e_{m}(n)-o_{m}(n)\big)>0
\end{align*}
for every $n$. This shows a striking difference with the case of even $m$, in which $e_{m}(n)-o_{m}(n)\in\{-1,0,1\}$ for all $n$ and the sign of the difference $e_{n}(n)-o_{m}(n)$ depends on the first non-zero digit in the $m$-ary representation of $n$. However, this is exactly the behaviour that we can expect. More precisely, if $2\nmid m$ then we need even number of powers of $m$ to obtain any partition of any even number $n$, and odd number of such powers to obtain a partition of odd $n$. In the case of even $m$, one can guess, that $o_{m}(n)$ and $e_{m}(n)$ have similar size, and this is very precisely described in Corollary \ref{CorDiffEvenOdd}.

\section{Complex roots}\label{SectionComplexRoots}

In the previous section we have considered the rational roots of the polynomials $p_{M}(n,t)$. Hence, we can move to a more general problem of studying the complex roots of these polynomials. This is much more difficult task and we will not characterize all of the roots. Instead, our aim is to prove that all the roots of $m$-ary partition polynomials lie in a small region on a complex plane. In fact, we will prove in Theorem \ref{ThmComplexRoots} that all such roots are in the circle of radius $4^{1/m}$ centered at the origin. We start this section by a heuristic argument explaining why we may expect such a result.

Recall that Rouche's theorem says that if $f$ and $g$ are two functions holomorphic in a neighbourhood of the closure of a bounded region $K$ with closed contour $\partial K$ without self-intersections, and $|f(z)-g(z)|<|f(z)|$ for $z\in\partial K$, then $f$ and $g$ have the same number of zeros inside $K$ (counting multiplicities). For more details see for example \cite{GK}.

In the case of $m$-ary partition polynomials, we have
\begin{align*}
p_{m}(n,t)-tp_{m}(n-1,t)=p_{m}(n/m,t),
\end{align*}
where $p_{m}(n/m,t)=0$ for all $t$ if $m\nmid n$. It follows that the difference $p_{m}(n,t)-tp_{m}(n-1,t)$ is a polynomial of much smaller degree than both, $p_{m}(n,t)$ and $tp_{m}(n-1,t)$. Therefore, we may expect that if $|t|$ is large enough, then $|p_{m}(n,t)-tp_{m}(n-1,t)|<|tp_{m}(n-1,t)|$. Obviously, there might be some unexpected cancellations even if $|t|$ is very large. However, assume for a while that we know that such cancellations do not occur if $|t|\geq R$ for some $R$. Then if we knew somehow (for example from some induction hypothesis) that all the zeroes of the polynomial $p_{m}(n-1,t)$ lay inside the circle
\begin{align*}
C(R):=\big\{\ z\in\mathbb{C}\ \big|\ |z|=R\ \big\},
\end{align*}
then we would get by Rouche's theorem that all the zeroes of $p_{m}(n,t)$ are inside $C(R)$ too. Therefore, we may expect, that all the zeroes of all polynomials $p_{m}(n,t)$ are contained inside a circle $C(R)$ for some $R$ independent of $m$ and $n$.

 %However, we will show that they do not occur if $|t|\geq 4^{1/m}$.

The above heuristic argument agrees with computational data. Moreover, more precise statement seems to be true. Ulas (in private communication) stated in the case of $m=2$ the following problem, that we extend here to the general case. 

%Let us introduce some notation. For a complex number $a$ and a non-empty set $B$ let
%\begin{align*}
%{\rm d} \left(a,B\right):=\inf\{\ |a-b|\ |\ b\in B\ \}.
%\end{align*}
%If sets $A$ and $B$ are non-empty subsets of $\mathbb{C}$ let
%\begin{align*}
%{\rm D} \left(A,B\right):=\sup\{\ {\rm d} \left(a,B\right) \ |\ a\in A,\ b\in B\ \}.
%\end{align*}

\begin{conj}\label{ConjComplexRoots}
Let $C(1)$ be the unit circle with the center in the origin. For all natural numbers $m\geq 2$ and $n$ let
\begin{align*}
Z_{m,n}:=\{\ z\in\mathbb{C}\ |\ p_{m}(n,z)=0\ \}\setminus \{0\}.
\end{align*}
Moreover, for $N\in\mathbb{N}$ let
\begin{align*}
Z_{m}^{(N)}:=\bigcup_{n>N}Z_{m,n}.
\end{align*}
Then, for every $\varepsilon >0$ and for every $m$ there exists $N$ (possibly depending on $\varepsilon$ and $m$) such that
\begin{align*}
Z_{m}^{(N)}\subseteq \big\{\ z \ \big|\ 1-\varepsilon< |z| <1+\varepsilon \ \big\}.
\end{align*}
%\begin{align*}
%{\rm D} \left(C(1),Z_{m}^{(N)}\right)<\varepsilon.
%\end{align*}

Moreover, we believe that for every $m$:
\begin{align*}
\#\left(Z_{m}^{(0)}\cap \big\{\ z\ \big|\ |z|>1\ \big\}\right)=\#\left(Z_{m}^{(0)}\cap \big\{\ z\ \big|\ |z|<1\ \big\}\right)=\infty.
\end{align*}
That is, we believe that there are infinitely many roots of $m$-ary partition polynomials inside and outside the unit circle.
\end{conj}

Let us show some computations to support Conjecture \ref{ConjComplexRoots}. We present all the roots of the polynomials $p_{4}(n,t)$ for $1\leq n\leq 400$ in Figure \ref{Figure4all} and all the roots of the polynomials $p_{5}(n,t)$ for $1\leq n\leq 500$ in Figure \ref{Figure5all}. 

%\begin{center}
\begin{figure}[H]
\centering
\begin{minipage}{0.45\textwidth}
\centering
\includegraphics[width=7.6cm, height=7.9cm]{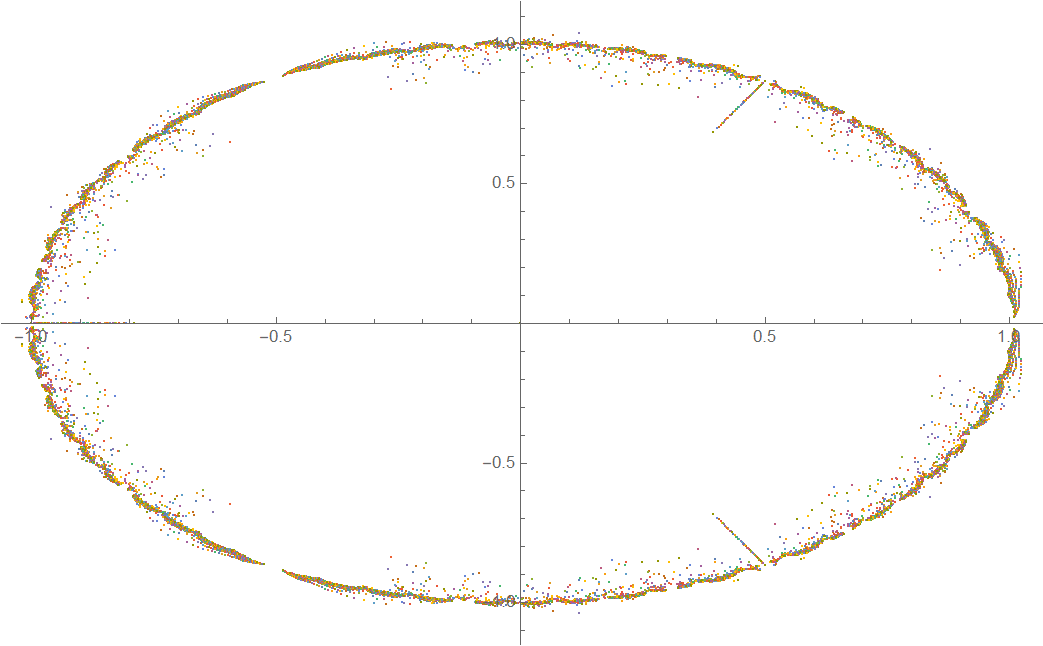}
\caption{All the roots of \\ the polynomials $p_{4}(n,t)$ for $1\leq n\leq 400$.}
\label{Figure4all}
%\end{figure}%\end{center}
%\begin{figure}[H]
\end{minipage}\hfill
\begin{minipage}{0.45\textwidth}
\centering
\includegraphics[width=7.6cm, height=7.9cm]{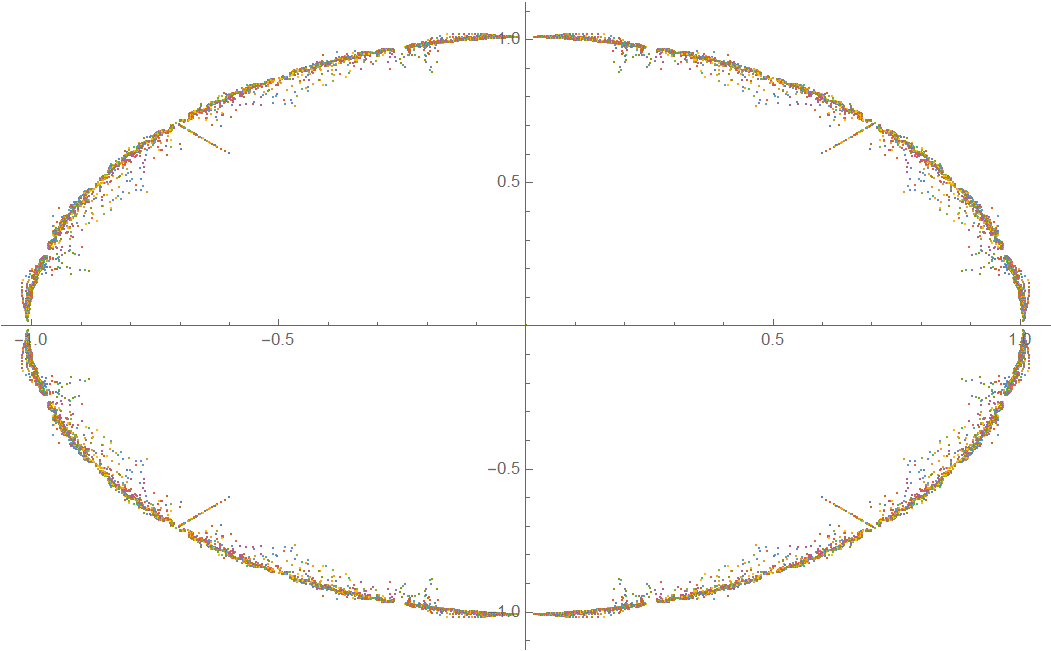}
\caption{All the roots of \\ the polynomials $p_{5}(n,t)$ for $1\leq n\leq 500$.}
\label{Figure5all}
\end{minipage}
\end{figure}

%We also provide pictures of all the roots of the polynomials $p_{4}(400,t)$ and $p_{5}(500,t)$ in Figure \ref{Figure45}. 

In fact, we made pictures of the roots of many polynomials $p_{m}(n,t)$ with small $m$ and large $n$ and all the pictures looked very similar to those that we present here.

%\begin{center}
%\begin{figure}[H]
%\centering
%    \begin{subfigure}[b]{\textwidth}
%        \centering
%        \includegraphics[width=8cm, height=8.2cm]{4}%
%        \hfill
%        \includegraphics[width=8cm, height=8.2cm]{5}
%        \caption{All the roots of the polynomials $p_{4}(400,t)$ and %$p_{5}(500,t)$.}
%        \label{Figure45}
%    \end{subfigure}
%\end{figure}
%\end{center}

We move to the main part of this section. Lemma \ref{LemComplexRoots} below gives precise description how and when the difference $|p_{m}(n,t)-tp_{m}(n-1,t)|$ is small compared to $|tp_{m}(n,t)|$.

\begin{lem}\label{LemComplexRoots}
Let $\varepsilon\in (0,1)$. If $|t|> \max\left\{\varepsilon^{-\frac{1}{m-1}},\big(\varepsilon (1-\varepsilon)\big)^{-\frac{1}{m}}\right\}$, then
\begin{align*}
|p_{m}(n,t)-tp_{m}(n-1,t)|<\varepsilon |tp_{m}(n-1,t)|.
\end{align*}
\end{lem}
\begin{proof}
If $m\nmid n$, then $|p_{m}(n,t)-tp_{m}(n-1,t)|=0$ and the statement is clearly true. Therefore, it is enough to show that
\begin{align}\label{IneqComplexRoots1}
|p_{m}(mn,t)-tp_{m}(mn-1,t)|<\varepsilon |tp_{m}(mn-1,t)|
\end{align}
for every $n$. For $n=1$ and for every $t$ such that $|t|> 1/\varepsilon^{1/(m-1)}$ we have
\begin{align*}
|p_{m}(m,t)-tp_{m}(m-1,t)|=|(t^{m}+t)-t\cdot t^{m-1}|=|t|<\varepsilon|t|^{m}=\varepsilon |tp_{m}(m-1,t)|.
\end{align*}

Let us assume that \eqref{IneqComplexRoots1} holds for all $k<n$. We want to prove it for $n$. Let us observe that the induction hypothesis (that is, \eqref{IneqComplexRoots1} with $k$ used instead of $n$) implies 
\begin{align}\label{IneqComplexRoots2}
|p_{m}(mk,t)|<\frac{1}{|t|^{m}(1-\varepsilon )}|p_{m}(m(k+1),t)|
\end{align}
for all $k\leq n-2$. Indeed, \eqref{IneqComplexRoots1} gives
\begin{align*}
|p_{m}(m(k+1),t)|\geq &\  |tp_{m}(m(k+1)-1,t)|-|p_{m}(m(k+1),t)-tp_{m}(m(k+1)-1,t)| \\ 
> &\ (1-\varepsilon )|tp_{m}(m(k+1)-1,t)|,
\end{align*}
and hence,
\begin{align*}
|p_{m}(mk,t)|= &\ \frac{1}{|t|^{m-1}}|p_{m}(m(k+1)-1,t)| \\ 
< &\ \frac{1}{|t|^{m-1}}\cdot \frac{1}{|t|(1-\varepsilon )} |p_{m}(m(k+1),t)|=\frac{1}{|t|^{m}(1-\varepsilon)}|tp_{m}(m(k+1),t)|.
\end{align*}

Let us now write $n=mn'+r$ for some $r\in\{0,\ldots ,m-1\}$. We apply \eqref{IneqComplexRoots2} $(n-1)-n'$ times and get
\begin{align*}
|p_{m}(mn',t)|<\frac{1}{|t|^{m(n-n'-1)}(1-\varepsilon )^{n-n'-1}}|p_{m}(m(n-1),t)|.
\end{align*}
Finally, we have the following chain of (in)equalities:
\begin{align*}
|p_{m}(mn,t)-tp_{m}(mn-1,t)|= &\ |p_{m}(n,t)|=|p_{m}(mn'+r,t)|=|t|^{r}|p_{m}(mn',t)| \\ 
< &\ \frac{|t|^{r}}{|t|^{m(n-n'-1)}(1-\varepsilon )^{n-n'-1}}|p_{m}(m(n-1),t)| \\
= &\ \frac{|t|^{r-(m-1)}}{|t|^{m(n-n'-1)}(1-\varepsilon )^{n-n'-1}}|p_{m}(m(n-1)+(m-1),t)| \\ 
\leq &\ \frac{1}{|t|^{m(n-n'-1)}(1-\varepsilon )^{n-n'-1}}|p_{m}(mn-1,t)|.
\end{align*}
In order to finish the proof, it is enough to check when the expression at the end of the last chain of inequalities is less than or equal to $\varepsilon |tp_{m}(mn-1,t)|$. This is equivalent to
\begin{align*}
\frac{1}{|t|^{m(n-n'-1)}(1-\varepsilon )^{n-n'-1}}\leq \varepsilon |t|.
\end{align*}
Since $|t|>1$, it is enough to check whether
\begin{align*}
\frac{1}{|t|^{m(n-n'-1)}(1-\varepsilon )^{n-n'-1}} \leq \varepsilon,
\end{align*}
that is,
\begin{align*}
|t|\geq \frac{1}{\varepsilon^{1/(m(n-n'-1))}(1-\varepsilon)^{1/m}}.
\end{align*}
We have $n-n'-1\geq 1$ so it is enough to have 
\begin{align*}
|t|\geq \frac{1}{\varepsilon ^{1/m}(1-\varepsilon )^{1/m}}.
\end{align*}
This finishes the proof.
\end{proof}

At this point, we can just conclude that the theorem below follows from Rouche's theorem, as we did in the heuristic argument at the beginning of this section. However, the inequality from Lemma \ref{LemComplexRoots} used in more explicit way provides the same information with an additional lower bound for $p_{m}(n,t)$ when $|t|\geq 4^{1/m}$.

\begin{theorem}\label{ThmComplexRoots}
If $t\in\mathbb{C}$ is such that $p_{m}(n,t)=0$, then $|t|< 4^{1/m}$.
\end{theorem}
\begin{proof}
From Lemma \ref{LemComplexRoots} with $\varepsilon =\frac{1}{2}$ we get that if $|t|>\max\{2^{1/(m-1)},4^{1/m}\}=4^{1/m}$, then
\begin{align*}
|p_{m}(n,t)|>\frac{1}{2}|tp_{m}(n-1,t)|
\end{align*}
for all $n$. In particular, if $|t|>4^{1/m}=2^{2/m}$ we have
\begin{align*}
|p_{m}(n,t)|>\frac{|t|^{n}}{2^{n}}|p_{m}(0,t)|=\frac{|t|^{n}}{2^{n}}>2^{\left(\frac{2}{m}-1\right)n}.
\end{align*}
Thus $p_{m}(n,t)\neq 0$ if $|t|>4^{1/m}$. Moreover, each polynomial $p_{m}(n,t)$ is a continuous function and we proved that the values of $p_{m}(n,t)$ for $|t|>4^{1/m}$ are bounded away from $0$ by a constant independent of $t$. Therefore, $p_{m}(n,t)\neq 0$ if $|t|=4^{1/m}$. The result follows.
\end{proof}

Now we can prove the following fact that we predicted at the beginning of this section.

\begin{cor}
For all $m\geq 2$, all roots of $m$-ary partition polynomials lay inside the circle of radius $2$ centered in the origin.
\end{cor}
\begin{proof}
By Theorem \ref{ThmComplexRoots} it is enough to observe, that $4^{1/m}\leq 4^{1/2}=2$ for $m\geq 2$.
\end{proof}

We finish our study of complex roots of $m$-ary partition polynomials by the following observation. The figures \ref{Figure4all}. and \ref{Figure5all} show some symmetries. Especially, some roots seem to lie on (rotated and scaled) $(m-1)$th roots of unity. We will explain this phenomenon in the next result.

\begin{theorem}\label{ThmPolynomialP}
\begin{enumerate}
\item Let $\widetilde{p_{m}}(n,t):=\frac{1}{t^{s_{m}(n)}}p_{m}(n,t)$. Then $\widetilde{p_{m}}(n,t)\in \mathbb{N}_{0}\left[t^{m-1}\right]$.
\item Let $P_{m}(n,q):=\widetilde{p_{m}}(mn,q^{1/(m-1)})$. Then $P_{m}(n,q)\in\mathbb{N}_{0}[q]$. Moreover, we have the following recurrence: $P_{m}(k)=\frac{q^{k+1}-1}{q-1}$ for $k\in\{0,\ldots ,m-1\}$, and
\begin{align*}
P_{m}(n)=q^{1+\nu_{m}(n)}P_{}(n-1,q) + P_{m}\left(\left\lfloor\frac{n}{m}\right\rfloor ,q\right)
\end{align*}
for $n\geq 2$, where $\nu_{m}(n)$ denotes the highest power of $m$ that divides $n$.
\end{enumerate}
\end{theorem}
\begin{proof}
From Example \ref{ExampleSmallCases} we immediately get
\begin{align*}
\widetilde{p_{m}}(km,t) = \frac{1}{t^{k}}p_{m}(km,t) = \frac{t^{(m-1)(k+1)}-1}{t^{m-1}-1},
\end{align*}
so indeed $P_{m}(k,q)=\frac{q^{k+1}-1}{q-1}$ for $k\in\{0,\ldots ,m-1\}$.

Recurrence relations \eqref{RecurRelationsmaryPartitions} imply
\begin{align*}
\frac{1}{t^{s_{m}(mn)}}p_{m}(mn,t) = t^{s_{m}(n-1)-s_{m}(n)+1} t^{m-1} \left(\frac{1}{t^{s_{m}(n-1)}} p_{m}(m(n-1),t)\right) + \frac{1}{t^{s_{m}(n)}}p_{m}(n,t),
\end{align*}
or
\begin{align*}
\widetilde{p_{m}}(mn,t) = t^{s_{m}(n-1)-s_{m}(n)+1} t^{m-1}\widetilde{p_{m}}(m(n-1),t) + \widetilde{p_{m}}(n,t).
\end{align*}

Let us study the quantity $s_{m}(n-1)-s_{m}(n)+1$. Let $n=\sum_{j=h}^{r}a_{j}m^{j}$, where $a_{h}\geq 1$. Then $h=\nu_{m}(n)$ and
\begin{align*}
n-1 & = \sum_{j=h+1}^{r}a_{j}m^{j} + (a_{h}-1)m^{h} + (m-1)m^{h-1}+\cdots + (m-1)m+(m-1).
\end{align*}
Therefore,
\begin{align*}
s_{m}(n-1)-s_{m}(n)+1 & = \sum_{j=h+1}^{r}a_{j} + (a_{h}-1) + (m-1)h -\sum_{j=h}^{r}a_{j}+1 \\
 & = (m-1)h = (m-1)\nu_{m}(n).
\end{align*}
Hence,
\begin{align*}
\widetilde{p_{m}}(mn,t) = t^{(m-1)(\nu_{m}(n)+1)}\widetilde{p_{m}}(m(n-1),t) + \widetilde{p_{m}}(n,t),
\end{align*}
and we get by a simple induction argument, that indeed $\widetilde{p_{m}}(mn,t)\in\mathbb{N}[t^{m-1}]$. By putting the definition of $P_{m}(n,q)$ into above equality, we immediately get the second part of the theorem.
\end{proof}

%%%%%%%%%%%%%%%%%%%%%%%%%%%%%%%%%%%%%%%%%%%%%%%%%%%%%%%%%%%%%%%%%%%%%%%%%%%%%%%%%%%%%%%%%%%%%%%%%%%%%%%%%%%%%%%%%%%%%%%%%%%%%%%%%%%%%

\section{Certain properties modulo another polynomials}\label{SectionCertainPropertiesModulo}

The aim of this section is to provide a polynomial generalization of the known characterisation of the number of $m$-ary partitions of $n$ modulo $m$ that depends on the $m$-ary representation of $n$. Recall, that if $n=a_{0}+a_{1}m+\cdots +a_{k}m^{k}$, then 
\begin{align}\label{Chacterisationmodm}
p_{m}(n,1)\equiv \prod_{j=1}^{k}(a_{j}+1)\pmod m.
\end{align}
The above result was independently proved by Alkauskas in \cite{Alk}, and Andrews, Freankel and Sellers in \cite{And1}. Another proofs were found by Edgar in \cite{Edgar} and \.Zmija in \cite{BZ5}.

Very interestingly, we will prove analogous characterisation for a more general class of $M$-ary partitions. Moreover, a polynomial version of this characterisation can be provided. We will do it in Theorem \ref{TheoremCharacterisationModulo}.

Recall that $M=(m_{n})_{n=0}^{\infty}$ is a sequence of integers such that $m_{0}=1$ and $m_{n}\geq 2$ for $n\geq 1$, and $M'=(1,m_{2},m_{3},\ldots )$. In other words, $M'$ is the sequence obtained from $M$ by omitting~$m_{1}$. 

We begin by the following technical lemma.

\begin{lem}\label{LemCharacterisationModuloM}
Let $n=km_{2}+r$, where $r\in\{0,1,\ldots ,m_{2}-1\}$, and $f(m,t):=\frac{t^{m}-1}{t-1}$ for an integer $m\geq 0$. Then
\begin{align*}
p_{M}(m_{1}n,t)\equiv t^{r}f(r+1,t^{m_{1}-1}) p_{M'}(km_{2},t) \pmod{t^{m_{1}+m_{2}-1}f(m_{2},t^{m_{1}-1})}.
\end{align*}
\end{lem}
\begin{proof}
At first, let us observe, that for every $k\leq n$ the following equality is true:
\begin{align}\label{EquLemEquiv}
p_{M}(m_{1}n,t)=t^{km_{1}}p_{M}(m_{1}(n-k),t)+\sum_{j=0}^{k-1}t^{jm_{1}}p_{M'}(n-j,t).
\end{align}
Indeed, the case $k=1$ follows immediately from equations \eqref{RecurRelationsMaryPartitions}, and if it is true for some $k$, then for $k+1$ we get:
\begin{align*}
p_{M}(m_{1}n,t)= &\ t^{km_{1}}p_{M}(m_{1}(n-k),t)+\sum_{j=0}^{k-1}t^{jm_{1}}p_{M'}(n-j,t) \\
= &\ t^{km_{1}}\big(t^{m_{1}}p_{M}(m_{1}(n-(k+1)),t)+p_{M'}(n-k,t)\big)+\sum_{j=0}^{k-1}t^{jm_{1}}p_{M'}(n-j,t) \\
= &\ t^{(k+1)m_{1}}p_{M}(m_{1}(n-(k+1)),t)+\sum_{j=0}^{k}t^{jm_{1}}p_{M'}(n-j,t).
\end{align*} 

In particular, for $k=n$ from \eqref{EquLemEquiv} we get:
\begin{equation}\label{EquLemEquality1}
\begin{aligned}
p_{M}(m_{1}n,t)= &\ t^{nm_{1}} p_{M}(0,t)+\sum_{j=0}^{n-1}t^{jm_{1}}p_{M'}(n-j,t) \\
= &\ \sum_{j=0}^{n}t^{jm_{1}}p_{M'}(n-j,t)=\sum_{j=0}^{n}t^{(n-j)m_{1}}p_{M'}(j,t).
\end{aligned}
\end{equation}

If we use \eqref{RecurRelationsMaryPartitions} for $M'$ instead of $M$, we get $p_{M'}(m_{2}n+j,t)=t^{j}p_{M'}(m_{2}n,t)$ for $j\in\{0,\ldots ,m_{2}-1\}$. Hence, for $l\in\{0,\ldots ,k-1\}$ we have
\begin{align*}
\sum_{j=lm_{2}}^{lm_{2}+m_{2}-1}t^{(n-j)m_{1}}p_{M'}(j,t)= &\ \sum_{j=0}^{m_{2}-1}t^{(n-lm_{2}-j)m_{1}}p_{M'}(lm_{2}+j,t) \\
= &\ \sum_{j=0}^{m_{2}-1}t^{(n-lm_{2}-j)m_{1}}\cdot t^{j} p_{M'}(lm_{2},t) \\
= &\ t^{(n-lm_{2})m_{1}}p_{M'}(lm_{2},t)\sum_{j=0}^{m_{2}-1}t^{(1-m_{1})j} \\
= &\ p_{M'}(lm_{2},t)t^{(n-lm_{2})m_{1}}\frac{1-t^{(1-m_{1})m_{2}}}{1-t^{1-m_{1}}} \\
= &\ p_{M'}(lm_{2},t)t^{(n-lm_{2})m_{1}}\cdot \frac{t^{(1-m_{1})m_{2}}}{t^{1-m_{1}}}\cdot\frac{t^{(m_{1}-1)m_{2}}-1}{t^{m_{1}-1}-1} \\
= &\ p_{M'}(lm_{2},t)t^{(n-lm_{2})m_{1}+(1-m_{1})m_{2}-(1-m_{1})}\frac{t^{(m_{1}-1)m_{2}}-1}{t^{m_{1}-1}-1}.
\end{align*}
Observe that
\begin{align*}
(n-lm_{2})m_{1}+(1-m_{1})m_{2}-(1-m_{1})\geq &\ (km_{2}-(k-1)m_{2})m_{1}-m_{1}m_{2}+m_{1}+m_{2}-1 \\
= &\ m_{1}+m_{2}-1,
\end{align*}
and thus
\begin{align}\label{EquLemEquality2}
\sum_{j=lm_{2}}^{lm_{2}+m_{2}-1}t^{(n-j)m_{1}}p_{M'}(j,t)\equiv 0\pmod{t^{m_{1}+m_{2}-1}f(m_{2},t^{m_{1}-1})}.
\end{align}

We perform similar computation in the case of $l=k$ and obtain
\begin{align*}
\sum_{j=km_{2}}^{km_{2}+r}t^{(n-j)m_{1}}p_{M'}(j,t) = &\ \sum_{j=0}^{r}t^{(n-km_{2}-j)m_{1}}p_{M'}(km_{2}+j,t) \\
= &\ t^{rm_{1}}p_{M'}(km_{2},t)\sum_{j=0}^{r}t^{(1-m_{1})j} \\
= & p_{M'}(km_{2},t)t^{rm_{1}}\frac{1-t^{(1-m_{1})(r+1)}}{1-t^{1-m_{1}}} \\
= &\ p_{M'}(km_{2})t^{rm_{1}+(1-m_{1})(r+1)-(1-m_{1})}\frac{t^{(m_{1}-1)(r+1)}-1}{t^{m_{1}-1}-1} \\
= &\ p_{M'}(km_{2},t) t^{r}\frac{t^{(m_{1}-1)(r+1)}-1}{t^{m_{1}-1}-1}=t^{r}f(r+1,t^{m_{1}-1}).
\end{align*}
The statement of our lemma simply follows from \eqref{EquLemEquality1}, \eqref{EquLemEquality2} and the above equality.
\end{proof}

In order to state the main theorem of this section we need to introduce some notation. For a positive integer $n$ let $(n_{0},\ldots ,n_{k})$ and $(a_{0},\ldots ,a_{k})$ be sequences constructed in the following way: 
\begin{align*}
\left\{\begin{array}{ll}
n_{0}:=\left\lfloor\frac{n}{m_{1}}\right\rfloor, & \\
n_{j}:=\left\lfloor\frac{n_{j-1}}{m_{j}}\right\rfloor, &  \textrm{ for } j\geq 1\\
a_{j-1}:=n_{j-1}-m_{j}n_{j}, & \textrm{ for } j\geq 1.
\end{array}\right.
\end{align*}
In other words, numbers $a_{j}$ can be seen as (unique) coefficients in the expansion
\begin{align*}
n=a_{0}+m_{1}(a_{1}+m_{2}(a_{2}+\cdots )))=a_{0}+a_{1}M_{1}+a_{2}M_{2}+\cdots +a_{k}M_{k},
\end{align*}
where $a_{j}\in \{0,\ldots ,m_{j+1}-1\}$. We call the above representation an $M$-ary representation of $n$. Let us prove, that this representation is indeed unique.

\begin{prop}\label{M-aryRepresentation}
For every $M$ and $n$ the coefficients in the expansion
\begin{align*}
n=a_{0}+a_{1}M_{1}+\cdots +a_{k}M_{k}
\end{align*}
are unique.
\end{prop}
\begin{proof}
We use induction on $n$. For every sequence $M$, if $0\leq n\leq m_{1}-1$, then the representation is obviously unique. Let us assume, that the representation is unique for all numbers less than some $n$ and all sequences $M$, and suppose that
\begin{align*}
n=a_{0}+a_{1}M_{1}+\cdots +a_{k}M_{k}=b_{0}+b_{1}M_{1}+\cdots +b_{l}M_{l}.
\end{align*}
All numbers $M_{j}$ for $j\geq 1$ are divisible by $m_{1}$. Therefore, the difference $a_{0}-b_{0}$ has to be divisible by $m_{1}$. But $-(m_{1}-1)\leq a_{0}-b_{0}\leq m_{1}-1$ so $a_{0}=b_{0}$. Hence we get
\begin{align*}
\frac{n-a_{0}}{m_{1}}=a_{1}M'_{1}+\cdots +a_{k}M'_{k}=b_{1}M'_{1}+\cdots +b_{l}M'_{l}.
\end{align*}
Now we can use the induction hypothesis to the number $(n-a_{0})/m_{1}$ and get $k=l$ and $a_{j}=b_{j}$ for all positive $j$'s. Hence, both of the above representations are equal. Therefore the number $n$ has only one representation. The proof is finished.
\end{proof}

%We defined the numbers $a_{j}$ uniquely by the above recurrence relation, but expression of the form $a_{0}+a_{1}M_{1}+a_{2}M_{2}+\cdots +a_{k}M_{k}$ with $a_{j}\in \{0,\ldots ,m_{j+1}-1\}$ may not be unique. Indeed, let us consider $M=(1,3,2)$. Then $M_{0}=1$, $M_{2}=3$ and $M_{3}=6$ and we can write number $7$ as
%\begin{align*}
%7=1+2\cdot 3 =1+1\cdot 6.
%\end{align*}
%It seems that the representation is unique if the sequence $M$ is non-decreasing.

For two monic polynomials $f_{1},f_{2}\in\mathbb{Z}[x]$ let $\gcd (f_{1},f_{2})$ be a polynomial of the highest possible degree that divides both $f_{1}$ and $f_{2}$. For polynomials $f_{1},\ldots ,f_{k}$ we define their greatest common divisor inductively as
\begin{align*}
\gcd (f_{1},\ldots ,f_{k}):=\gcd (\gcd (f_{1},\ldots ,f_{k-1}),f_{k}).
\end{align*}

We are ready to state the main theorem of this section of the paper.

\begin{theorem}\label{TheoremCharacterisationModulo}
For a positive integer $h$ let us define
\begin{align*}
g_{h}(t):=\gcd \big(t^{m_{1}+m_{2}-1}f(m_{2},t^{m_{1}-1}), \ldots ,t^{m_{h}+m_{h+1}-1}f(m_{h+1},t^{m_{h}-1})\big).
\end{align*}
Let $n\in\mathbb{N}$ and $n=a_{0}+a_{1}M_{1}+\cdots +a_{k}M_{k}$ be its $M$-ary representation. Then
\begin{align*}
p_{M}(n,t)\equiv t^{a_{0}}\prod_{j=1}^{k}t^{a_{j}}f(a_{j}+1,t^{m_{j}-1}) \pmod{ g_{k}(t)} .
\end{align*}
\end{theorem}
\begin{proof}
We use an induction on $k$. For $k=0$, that is, if $n=a_{0}$, the statement follows from \eqref{RecurRelationsMaryPartitions}. Assume that it is true for all sequences $M$ and all natural numbers with at most $k$ coefficients in their $M$-ary representations. Let $n=a_{0}+a_{1}M_{1}+\cdots +a_{k}M_{k}$, that is, $n$ has $k+1$ coefficients. Let $n_{1}$ and $n_{2}$ be such that $n=n_{1}m_{1}+a_{0}$ and $n_{1}=n_{2}m_{2}+a_{1}$. Then Lemma \ref{LemCharacterisationModuloM} and the induction hypothesis (used for $M'$ instead of $M$ and $n_{1}$ instead of $n$) implies
\begin{align*}
p_{M}(n,t)= &\ p_{M}(n_{1}m_{1}+a_{0},t)=t^{a_{0}}p_{M}(n_{1}m_{1},t)\equiv t^{a_{0}}\cdot t^{a_{1}}f(a_{1}+1,t^{m_{1}-1})p_{M'}(n_{2}m_{2},t) \\
\equiv &\ t^{a_{0}}\cdot t^{a_{1}}f(a_{1}+1,t^{m_{1}-1}) \cdot \prod_{j=2}^{k}t^{a_{j}}f(a_{j}+1,t^{m_{j}-1}) \pmod{g_{k}(t)}.
\end{align*}
This finishes the proof.
\end{proof}

The above theorem has interesting consequences. We list some of them.

\begin{cor}\label{CorCharModm}
Let $n=n_{0}+n_{1}m+\cdots +n_{s}m^{s}$ be the representation of $n$ in base $m$. For a natural number $k$ let us consider the polynomial
\begin{align*}
    f_{k}(t):=\frac{t^{(m-1)k}-1}{t^{m-1}-1}.
\end{align*}
Then
\begin{align*}
    p_{m}(n,t)\equiv t^{n_{0}}\prod_{j=1}^{s}t^{n_{j}}f_{n_{j}+1}(t)\pmod{t^{2m-1}f_{m}(t)}.
\end{align*}
\end{cor}
\begin{proof}
Put $M=(1,m,m,\ldots )$ in Theorem \ref{TheoremCharacterisationModulo}.
\end{proof}

\begin{cor}\label{CorCharModg}
Let $n=a_{0}+a_{1}M_{1}+\cdots +a_{k}M_{k}$ be the representation of $n$ in base $M$. Then
\begin{align*}
p_{M}(n)\equiv \prod_{j=1}^{k}(a_{j}+1) \pmod{\gcd (m_{1},\ldots ,m_{k})}.
\end{align*}
\end{cor}
\begin{proof}
Put $t\to 1$ in Theorem \ref{TheoremCharacterisationModulo}.
\end{proof}

At the end note, that the statement of Corollary \ref{CorCharModm} looks especially elegant if we use the polynomials $P_{m}(n,t)$ introduced in Theorem \ref{ThmPolynomialP}, and some notation from so-called quantum calculus. Roughly speaking, quantum calculus is a theory in which we replace classical definition of derivative of a function $f(x)$ by its $q$-analogue: 
\begin{align*}
D_{q}f(x):=\frac{f(qx)-f(x)}{(q-1)x}.
\end{align*}

In quantum calculus, the role of positive integers $n$ is played by the following polynomials in~$q$:
\begin{align*}
[n]_{q}:=\frac{q^{n}-1}{q-1}.
\end{align*}
Indeed, we have for example that then $D_{q}(x^{n})=[n]_{q}x^{n-1}$. For more details see \cite{ChK}.

With the above notation in mind, we can rephrase Corollary \ref{CorCharModm} in the following very elegant way.

\begin{cor}
Let $n=a_{1}+a_{2}m+\cdots +a_{k}m^{k-1}$. Then
\begin{align*}
P_{m}(n,q)\equiv \prod_{j=1}^{k}[a_{j}+1]_{q} \pmod{[m]_{q}}.
\end{align*}
\end{cor}

In other words, our result may be seen as a quantum generalisation of the characterisation \eqref{Chacterisationmodm} of the sequence of $m$-ary partitions modulo $m$.

\section{Coefficients of \texorpdfstring{$M$}{TEXT}-ary partition polynomials}

We devote the next part of this paper to study of the coefficients of polynomials $p_{M}(n,t)$. Let us write
\begin{align*}
p_{M}(n,t)=\sum_{j=0}^{n}a_{M}(j,n)t^{j}.
\end{align*}
Recall that then $a_{M}(j,n)$ is equal to the number of $M$-ary partitions of $n$ with exactly $j$ parts.

From relations \eqref{RecurRelationsMaryPartitions} we have
\begin{align}\label{EquRelationaM}
a_{M}(j,n)=a_{M}(j-1,n-1)+a_{M'}(j,n/m_{1}),
\end{align}
where we use a convention that $a_{M}(i,q)=0$ if $i>q$ or $q\not\in\mathbb{N}$. 

Our aim is to generalize results from \cite{BZ2}, where the case of binary partition polynomials (that is, $m=2$) was considered. We will study two types of sequences. At first, we will narrow aour considerations down to the $m$-ary case and focus on the sequence $(a_{M}(k,n))_{n=0}^{\infty}$. This is the sequence of coefficients of $t^{k}$ of consecutive $m$-ary partition polynomials, where $k$ is a fixed natural number. In particular, we will prove that for each $k$ this sequence is bounded. It is important to note that in \cite{BZ2} the following bound has been showed in the case of binary partition polynomials: $a_{2}(j,n)\leq j!$ for all $n$. However, there was a mistake in the proof. Fortunately, we use a bit different approach here and prove a better bound in Theorem \ref{Thma_m(k,n)UpperBound} below. More details in the case of $m=2$ are provided in the remark after Theorem \ref{Thma_m(k,n)UpperBound}. % We also describe how to obtain a much better bound in the remark after Theorem \ref{Thma_m(k,n)UpperBound}.

The second sequence that will be interesting for us is $(a_{M}(n-k,n))_{n=0}^{\infty}$, that is, sequence of coefficients of $t^{n-k}$ in polynomials $p_{M}(n,t)$, where $k$ is again a fixed positive integer. We will prove that this sequence is eventually constant and relate this constant value to another kind of partitions, that is, partitions with parts of the form $M_{j}-1$. This relation was found in the special case when $M_{j}=2^{j}$ for all $j\geq 1$ in \cite{BZ2}. The sequence counting partitions into parts of the form $2^{j}-1$ is known as a sequence of $s$-partitions and was introduced by Bhatt in \cite{Bhatt}. He used some properties of this sequence in his algorithm that computes expressions of the form $a^{n}\mod{m}$. Another properties of the sequence of $s$-partitions were studied in \cite{Goh}.

Let us begin with two lemmas. Lemma \ref{Lema_m(j,n)=0} is especially interesting because it implies that for every $j$, there are infinitely many zeroes in the sequence $(a_{m}(j,n))_{n=0}^{\infty}$.

\begin{lem}\label{Lema_m(j,n)=0}
Let $j$, $r\geq 1$, and let $m\nmid n$. If $a_{m}(j,m^{r}n-1)\geq 1$, then $j\geq (m-1)r$.
\end{lem}
\begin{proof}
Let us assume that there is a representation
\begin{align*}
m^{r}n-1=m^{s_{1}}+\cdots +m^{s_{j}},
\end{align*}
for some non-negative integers $s_{1}\leq\ldots\leq s_{j}$. If we have $s_{l}=\ldots =s_{l+m-1}$ for some $l$, we can replace these $m$ powers of $m$ by a single power $m^{s_{l}+1}$. We then get another representation with possibly smaller number of parts. We can perform this procedure as many times as possible and get representation of the form
\begin{align*}
m^{r}n-1=m^{t_{1}}+\cdots +m^{t_{j'}}
\end{align*}
with $t_{1}\geq \ldots \geq t_{j'}$, where each power occurs at most $m-1$ times, and $j'\leq j$. Consider the above equality modulo $m^{r}$, that is,
\begin{align*}
-1\equiv m^{t_{1}}+\cdots +m^{t_{l}} \pmod{m^{r}},
\end{align*}
where $1\leq t_{1}\leq \ldots\leq t_{l}\leq r-1$ and $j'\geq l$. Observe, that
\begin{align*}
m^{t_{1}}+\cdots +m^{t_{l}}\leq (m-1)(1+\cdots +m^{r-1})=m^{r}-1<m^{r}.
\end{align*}
Hence,
\begin{align*}
m^{r}-1=m^{t_{1}}+\cdots +m^{t_{l}},
\end{align*}
and because this is the largest possible value of the expression on the right-hand side, we need to have $l=(m-1)r$. Thus $j\geq j'\geq l=(m-1)r$, and we are done.
\end{proof}

%\begin{lem}\label{Lema_m(k,n)Rekur}
%Let $a_{0}\in\{0,\ldots ,m-1\}$. Then
%\begin{align*}
%a_{m}(k,mn_{1}+a_{0})=\sum_{0\leq i\leq \min\{(k-a_{0})/m , n_{1}\}} a_{m}\left(k-a_{0}-im, n_{1}-i\right).
%\end{align*}
%\end{lem}
%\begin{proof}
%We prove that for every $N\geq 1$ we have
%\begin{align}\label{EquLema_m(k,n)Rekur}
%a_{m}(k,mn_{1}+a_{0})=a_{m}(k-a_{0}-Nm,m(n_{1}-N))+\sum_{i=0}^{N-1}a_{m}(k-a_{0}-im,n_{1}-i).
%\end{align}
%Indeed, for $N=1$ this is exactly \eqref{EquRelationaM}, and if it is true for some $N$, then for $N+1$ we get using \eqref{EquRelationaM}:
%\begin{align*}
%a_{m}(k,mn_{1}&+a_{0})= a_{m}(k-a_{0}-Nm,m(n_{1}-N))+\sum_{i=0}^{N-1}a_{m}(k-a_{0}-im,n_{1}-i) \\ 
%= &\ a_{m}(k-a_{0}-Nm-1, m(n_{1}-N)-1)+a_{m}(k-a_{0}-Nm,n_{1}-N)+\sum_{i=0}^{N-1}a_{m}(k-a_{0}-im,n_{1}-i) \\
%= &\ a_{m}(k-a_{0}-Nm-m, m(n_{1}-N)-m)+\sum_{i=0}^{N}a_{m}(k-a_{0}-im,n_{1}-i) \\
%= &\ a_{m}(k-a_{0}-(N+1)m, m(n_{1}-(N+1)))+\sum_{i=0}^{N}a_{m}(k-a_{0}-im,n_{1}-i).
%\end{align*}
%Hence \eqref{EquLema_m(k,n)Rekur} is true. In particular, if $N\geq \min\left\{\frac{k-a_{0}}{m},n_{1}\right\}$ then $a_{m}(k-a_{0}-(N+1)m, m(n_{1}-(N+1)))=0$. Therefore the result follows by putting $N=\min\left\{\frac{k-a_{0}}{m},n_{1}\right\}$.
%\end{proof}

\begin{lem}\label{Lema_m(k,n)Rekur}
Let $n$, $r\geq 0$ be such that $m\nmid n$. Then the following equality holds:
\begin{align*}
a_{m}(k,m^{r}n)=\sum_{j=0}^{r-1}a_{m}(k-1,m^{r-j}n-1)+a_{m}(k,n).
\end{align*}
\end{lem}
\begin{proof}
We prove that for every $N\geq 0$ such that $N\in\{0,\ldots ,r-1\}$ we have:
\begin{align*}
a_{m}(k,m^{r}n)=\sum_{j=0}^{N}a_{m}(k-1,m^{r-j}n-1)+a_{m}(k,m^{r-N-1}n).
\end{align*}
Indeed, for $N=0$ this is exactly \eqref{EquRelationaM}, and if it is true for some $N$, then for $N+1$, using \eqref{EquRelationaM}, we get the following chain of equalities:
\begin{align*}
a_{m}(k,m^{r}n)= & \sum_{j=0}^{N}a_{m}(k-1,m^{r-j}n-1)+a_{m}(k,m^{r-N-1}n) \\
= & \sum_{j=0}^{N}a_{m}(k-1,m^{r-j}n-1)+a_{m}(k-1,m^{r-N-1}n)+a_{m}(k,m^{r-N-1-1}n) \\
= & \sum_{j=0}^{N+1}a_{m}(k-1,m^{r-j}n-1) + a_{m}(k,m^{r-(N+1)-1}n).
\end{align*}
The result follows by putting $N=r-1$.
\end{proof}

\begin{theorem}\label{Thma_m(k,n)UpperBound}
Let $k\in\mathbb{N}$ and $n\equiv i\pmod{m}$ for some $i\in\{0,\ldots ,m-1\}$. Then the following inequality holds:
\begin{align*}
a_{m}(k,n)\leq \left(\frac{m}{m-1}\right)^{\lfloor (k-i)/m\rfloor}\ \ \prod_{j=0}^{\lfloor (k-i)/m\rfloor-1}\left(k-i-jm\right).
\end{align*}
In particular, the sequence $(a_{m}(k,n))_{n=0}^{\infty}$ is bounded.
\end{theorem}
\begin{proof}
We want to find for $i\in\{0,\ldots ,m-1\}$ sequences $c_{k}^{(i)}$ such that for every $n$ if $n\equiv i\pmod{m}$ then $a_{m}(k,n)\leq c_{k}^{(i)}$. For $k=0$ and $k=1$ we have $a_{m}(0,n)=1$ and
\begin{align*}
a_{m}(1,n)=\left\{\begin{array}{ll}
1, & \textrm{if } n \textrm{ is a power of } m, \\
0, & \textrm{otherwise}.
\end{array}\right.
\end{align*}
Hence $a_{m}(1,n)\leq 1$ and we can define $c_{1}^{(i)}=1$ for every $i$. 

Now we want to express every $c_{k}^{(i)}$ in terms of the numbers $c_{k-1}^{(j)}$. If $n=mn_{1}+a_{0}$ for some $a_{0}\in\{1,\ldots ,m-1\}$, then by \eqref{EquRelationaM} we get
\begin{align*}
a_{m}(k,n)=a_{m}(k-1,n-1)\leq c_{k-1}^{(a_{0}-1)},
\end{align*}
so we can take $c_{k}^{(a_{0})}:=c_{k-1}^{(a_{0}-1)}$.

Let us now assume that $n=m^{r_{1}}n_{1}$ where $m\nmid n_{1}$. Observe, that for some $r\leq \frac{k}{m-1}$ we have
\begin{align*}
a_{m}(k,m^{r_{1}}n_{1})=a_{m}(k,m^{r}n_{1}).
\end{align*}
Indeed, if $r_{1}\leq \frac{k}{m-1}$, then we can take $r=r_{1}$. Otherwise, by \eqref{EquRelationaM} and Lemma \ref{Lema_m(j,n)=0} (used with $j=k-1$) we have
\begin{align*}
a_{m}(k,m^{r_{1}}n_{1})=a_{m}(k-1,m^{r_{1}}n_{1}-1)+a_{m}(k,m^{r_{1}-1}n_{1})=a_{m}(k,m^{r_{1}-1}n_{1}).
\end{align*}
Now, if $r_{1}-1\leq k/(m-1)$ we take $r=r_{1}-1$. Otherwise, we repeat the procedure.

We assume that $n_{1}\equiv i\pmod{m}$ for some $i\in\{1,\ldots ,m-1\}$. Finally, we use Lemma \ref{Lema_m(k,n)Rekur} to get %and the induction hypothesis to get
\begin{align*}
a_{m}(k,m^{r_{1}}n_{1}) & = a_{m}(k,m^{r}n_{1})=\sum_{j=0}^{r-1}a_{m}(k-1,m^{r-j}n_{1}-1)+a_{m}(k,n_{1}) \\
& \leq \sum_{j=0}^{r-1}c_{k-1}^{(m-1)}+c_{k}^{(i)}=rc_{k-1}^{(m-1)} + c_{k-1}^{(i-1)} \leq \left\lfloor\frac{k}{m-1}\right\rfloor c_{k-1}^{(m-1)} + c_{k-1}^{(i-1)}.
\end{align*}
Suppose further that the numbers $c_{j}^{(i)}$ satisfy $c_{j}^{(i)}\leq c_{j}^{(0)}$ for every $j$ and every $i\in\{0,\ldots ,m-1\}$. Therefore, we can define
\begin{align*}
c_{k}^{(0)} := \left\lfloor\frac{k}{m-1}\right\rfloor c_{k-1}^{(m-1)} + c_{k-1}^{(0)}.
\end{align*}

Using the information obtained so far, we can conclude that if $n\equiv i\pmod{m}$ then $a_{m}(k,n)\leq c_{k}^{(i)}=c_{k-i}^{(0)}$. Therefore, in order to finish the proof, we need to find upper bounds for the terms in the sequence $C_{k}:=c_{k}^{(0)}$, which satisfies the following recurrence relation: $C_{j}=1$ for $j\in\{0,\ldots ,m-1\}$, and $C_{k} = C_{k-1} + \left\lfloor\frac{k}{m-1}\right\rfloor C_{k-m}$ for $k\geq m$. Of course, the sequence $(C_{k})_{k=0}^{\infty}$ is then increasing, so
\begin{align*}
C_{k} & =C_{k-1} + \left\lfloor\frac{k}{m-1}\right\rfloor C_{k-m} = C_{k-2} + \left\lfloor\frac{k-1}{m-1}\right\rfloor C_{k-1-m} + \left\lfloor\frac{k}{m-1}\right\rfloor C_{k-m} \\
 & \leq C_{k-2} + \left(\left\lfloor\frac{k-1}{m-1}\right\rfloor + \left\lfloor\frac{k}{m-1}\right\rfloor\right) C_{k-m} \\
 & = C_{k-3} + \left\lfloor\frac{k-2}{m-1}\right\rfloor C_{k-2-m} + \left(\left\lfloor\frac{k-1}{m-1}\right\rfloor + \left\lfloor\frac{k}{m-1}\right\rfloor\right) C_{k-m} \\
 & < C_{k-3} + \left(\left\lfloor\frac{k-2}{m-1}\right\rfloor + \left\lfloor\frac{k-1}{m-1}\right\rfloor + \left\lfloor\frac{k}{m-1}\right\rfloor\right) C_{k-m} \\
 & < C_{k-m} +\left(\left\lfloor\frac{k-m+1}{m-1}\right\rfloor+\cdots + \left\lfloor\frac{k}{m-1}\right\rfloor\right) C_{k-m} \\
 & = \left(1+ \left\lfloor\frac{k}{m-1}\right\rfloor -1 + \left\lfloor\frac{k-m+2}{m-1}\right\rfloor+\cdots +\left\lfloor\frac{k}{m-1}\right\rfloor\right) C_{k-m} \\
 & \leq m\left\lfloor\frac{k}{m-1}\right\rfloor C_{k-m}\leq \frac{mk}{m-1} C_{k-m}.
\end{align*}
The above inequality implies, by a simple induction argument, that for every $t\leq k/m$ we have
\begin{align*}
C_{k} < \left(\frac{m}{m-1}\right)^{t}\ \ \prod_{j=0}^{t-1}\left(k-jm\right) C_{k-tm}.
\end{align*}
In particular, for $t=\lfloor k/m\rfloor$, we get
\begin{align*}
C_{k} < \left(\frac{m}{m-1}\right)^{\lfloor k/m\rfloor}\ \ \prod_{j=0}^{\lfloor k/m\rfloor-1}\left(k-jm\right).
\end{align*}
The result follows.
\end{proof}

\begin{rem}
Observe, that the expression from the bound from Theorem \ref{Thma_m(k,n)UpperBound} is decreasing in $m$, so we can find a bound which is uniform in $m$ and $n$:
\begin{align*}
a_{m}(k,n)\leq 2^{\lfloor k/2\rfloor} \ \ \prod_{j=0}^{\lfloor k/2\rfloor -1} (k-2j) = \left\{ \begin{array}{ll}
2^{k} \left(\frac{k}{2}\right)!, & \textrm{if } k \textrm{ is even}, \\
\frac{k!}{\left(\frac{k-1}{2}\right)!}, & \textrm{if } k \textrm{ is odd}.
\end{array} \right.
\end{align*}
This significantly improves the bound $a_{2}(2k+i,n)\leq (2k+i)!$ for $i\in\{0,1\}$ from \cite{BZ2}.
\end{rem}

In order to further study the sequence $(a_{M}(k,n))_{n=0}^{\infty}$, let us define
\begin{align*}
A_{M}(k,x):=\sum_{n=0}^{\infty}a_{M}(k,n)x^{n}.
\end{align*}
More precisely, we are now interested in the sequence of coefficients of $x^{k}$ in consecutive polynomials $p_{m}(n,t)$, where $k$ is fixed and $n$ varies. Our aim is to generalize results from \cite{BZ2} related to the coefficients of binary partition polynomials.

From \eqref{EquRelationaM} we get
\begin{align*}%\label{RekurA_M(k,x)}
A_{M}(k,x)-xA_{M}(k-1,x)=A_{M'}(k,x^{m_{1}}).
\end{align*}

We focus on the $m$-ary case, so in particular $M=M'$. The above relation simplifies to
\begin{align}\label{RekurA_m(k,x)}
A_{m}(k,x)-xA_{m}(k-1,x)=A_{m}(k,x^{m}),
\end{align}
where we write $A_{m}(k,x)$ instead of $A_{M}(k,x)$ in this case.

Our aim is to show that for every $m$ and $k$, the function $A_{m}(k,x)$ satisfies a Mahler type functional equation, that is, an equation of the form
\begin{align*}
\sum_{j=0}^{t}f_{j}(x)A_{m}\left(k,x^{m^{j}}\right)+g(x)=0
\end{align*}
for some polynomials $f_{j}(x)$ and $g(x)$, and a positive integer $t$.

\begin{theorem}\label{ThmMahlerA_m}
Let $m\geq 2$ and $k\geq 1$ be fixed. The function $A_{m}(k,x)$ satisfies the equation
\begin{align*}
\sum_{j=0}^{k}P_{k,j}(x)A_{m}\left(k,x^{m^{j}}\right) +Q_{k}(x)=0.
\end{align*}
Here $P_{k,j}(x)$ and $Q_{k}(x)$ are polynomials with integer coefficients satisfying the following recurrence relations:
\begin{align*}
\left\{\begin{array}{ll}
P_{0,0}(x)=1, & \\
Q_{0}(x)=-1, & \\
P_{k,0}(x)=x^{m^{k-1}-1}P_{k-1,0}(x), & \textrm{for } k\geq 1, \\
P_{k,j}(x)=x^{m^{k}-m^{j}}P_{k-1,j}(x)-x^{m^{k-1}-m^{j-1}}P_{k-1,j-1}(x), & \textrm{for } k\geq 1 \textrm{ and } j\in\{1,\ldots ,k-1\}, \\
P_{k,k}(x)=-P_{k-1,k-1}(x), & \\
Q_{k}(x)=x^{m^{k-1}}Q_{k-1}(x).
\end{array}\right.
\end{align*}
\end{theorem}
\begin{proof}
We use an induction argument. For $k=1$ we get, from \eqref{RekurA_m(k,x)} the following equation:
\begin{align*}
0=A_{m}(1,x)-xA_{m}(0,x)-A_{m}(1,x^{m})=A_{m}(1,x)-A_{m}(1,x^{m})-x.
\end{align*}
Assume that the statement is true for some $k-1$. That is, we assume that
\begin{align*}
\sum_{j=0}^{k-1}P_{k-1,j}(x)A_{m}(k-1,x^{j}) + Q_{k-1}(x)=0,
\end{align*}
or equivalently,
\begin{align}\label{EquMahlerA_m}
A_{m}(k-1,x)=-\frac{\sum_{j=1}^{k-1}P_{k-1,j}(x)A_{m}(k-1,x^{j})+Q_{k-1}(x)}{P_{k-1,0}(x)}.
\end{align}

Relation \eqref{RekurA_m(k,x)} implies that for every $j\geq 0$ we have
\begin{align}\label{EquMahlerA_m2}
A_{m}(k-1,x^{m^{j}})=\frac{A_{m}(k,x^{m^{j}})-A_{m}(k,x^{m^{j+1}})}{x^{m^{j}}}.
\end{align}

Using \eqref{RekurA_m(k,x)}, \eqref{EquMahlerA_m} and \eqref{EquMahlerA_m2}, we get
\begin{align*}
A_{m}(k,x)-A_{m}(k,x^{m})= &\ xA_{m}(k-1,x) \\
= &\ -\frac{x}{P_{k-1,0}(x)}\left(\sum_{j=1}^{k-1}P_{k-1,j}(x)A_{m}(k-1,x^{j})+Q_{k-1}(x)\right) \\
= &\ -\frac{x}{P_{k-1,0}(x)}\left(\sum_{j=1}^{k-1}P_{k-1,j}(x)\frac{A_{m}(k,x^{m^{j}})-A_{m}(k,x^{m^{j+1}})}{x^{m^{j}}}+Q_{k-1}(x)\right).
\end{align*}
Thus
\begin{align*}
x^{m^{k-1}-1}P_{k-1,0}(x)A_{m}(k,x)+  \sum_{j=0}^{k-1} & \left(x^{m^{k}-m^{j}}P_{k-1,j}(x)-x^{m^{k-1}-m^{j-1}}P_{k-1,j-1}(x)\right)A_{m}(k,x^{m^{j}}) \\
- & P_{k-1,k-1}(x)A_{m}(k,x^{m^{k}})+x^{m^{k-1}}Q_{k-1}(x)=0,
\end{align*}
and the result follows.
\end{proof}

\begin{cor}
Let $m=p$ be a prime number. For every $k$ the sequence $(a_{p}(k,n) \Mod p)_{n=0}^{\infty}$ is $p$-automatic.
\end{cor}
\begin{proof}
Theorem \ref{ThmMahlerA_m} implies that the function $A_{p}(k,x)$ satisfies the equation
\begin{align*}
\sum_{j=0}^{k}P_{k,j}(n)A_{p}(k,x)^{p^{j}}+Q_{k}(x)\equiv 0 \pmod p.
\end{align*}
In other words, $A_{p}(k,x)$ is algebraic over $\mathbb{F}_{p}$. The result follows from Christol's Theorem (Lemma \ref{ThmChristol}).
\end{proof}

In the case of $m$ not being a prime we cannot use Christol's Theorem and get the analogous result. Therefore, we only state the following question for possible future work.

\begin{ques}
Is the sequence $(a_{m}(k,n) \Mod m)_{n=0}^{\infty}$ $m$-automatic for every $k$?
\end{ques}

Let us move to study the coefficients of $t^{n-k}$ in the $n$th $M$-ary polynomial, where $k$ is a fixed positive integer. Observe that the difference between two consecutive $M$-ary partition polynomials, where the former one is multiplied by t, is always a polynomial of much smaller degree. This suggests that for every $k$ the value of $a_{M}(n-k,n)$ is constant for sufficiently large values of $n$. We prove this statement in the next theorem.

\begin{theorem}\label{Thma(n-k,n)}
Let $k\in\mathbb{N}$. Then for every $n\geq \frac{km_{1}}{m_{1}-1}$ we have
\begin{align*}
a_{M}(n-k,n)= \sum_{k/m_{1}\leq j \leq k/(m_{1}-1)}a_{M'}(m_{1}j-k,j) .
\end{align*}
In particular, if $n\geq km_{1}/(m_{1}-1)$ and $1\leq k\leq m_{1}-2$, then $a_{M}(n-k,n)=0$.
\end{theorem}
\begin{proof}
Let
\begin{align*}
G_{k}(x)=\sum_{n=0}^{\infty}a_{M}(n-k,n)x^{n},
\end{align*}
and observe that \eqref{EquRelationaM} implies
\begin{align*}
G_{k}(x)= & \sum_{n=0}^{\infty}\big(a_{M}(n-k-1,n-1)+a_{M'}(n-k,n/m_{1})\big)x^{n} \\
= & x\sum_{n=1}^{\infty}a_{M}((n-1)-k,n-1)x^{n-1} + \sum_{n=0}^{\infty}a_{M'}(n-k, n/m_{1})x^{n} \\
= & xG_{k}(x)+\sum_{k\leq n\leq km_{1}/(m_{1}-1)}a_{M'}(n-k,n/m_{1})x^{n},
\end{align*}
because $a(n-k,n/m_{1})=0$ if $n-k>n/m_{1}$, that is, if $n>km_{1}/(m_{1}-1)$. Therefore,
\begin{align*}
G_{k}(x)= & \frac{1}{1-x}\sum_{k\leq n\leq km_{1}/(m_{1}-1)}a_{M'}(n-k,n/m_{1})x^{n} \\
 = &\ \sum_{n=0}^{\infty}\left(\sum_{j=0}^{n}a_{M'}(j-k,j/m_{1})\right)x^{n} \\
= & \sum_{k\leq n<km_{1}/(m_{1}-1)}\left(\sum_{j=0}^{n}a_{M'}(j-k,j/m_{1})\right)x^{n} \\
 &\ \hspace{2cm} + \sum_{n\geq km_{1}/(m_{1}-1)}\left(\sum_{k\leq j\leq km_{1}/(m_{1}-1)}a_{M'}(j-k,j/m_{1})\right)x^{n}.
\end{align*}
Hence, if $n\geq km_{1}/(m_{1}-1)$, then
\begin{align*}
a_{M}(n-k,n)= & \sum_{k\leq j\leq km_{1}/(m_{1}-1)}a_{M'}(j-k,j/m_{1}) = \sum_{k/m_{1}\leq j \leq k/(m_{1}-1)}a_{M'}(m_{1}j-k,j).
\end{align*}
This finishes the first part of our theorem. For the second part, it is enough to notice that if $k<m_{1}-1$, then
\begin{align*}
0<\frac{k}{m_{1}}<\frac{k}{m_{1}-1}<1,
\end{align*}
so the sum in the expression defining $a_{M}(n-k,n)$ is empty. The result follows.
\end{proof}

It turns out that the constant value of $a_{M}(n-k,n)$ for $n\geq km_{1}/(m_{1}-1)$ is equal to the number of some another kind of partitions of $n$. More precisely, let
\begin{align*}
S_{M}(x):=\prod_{n=1}^{\infty}\frac{1}{1-x^{M_{n}-1}}=\sum_{n=0}^{\infty}s_{M}(n)x^{n}.
\end{align*}
That is, for every $n$ the number $s_{M}(n)$ is equal to the number of partitions of $n$ into parts of the form $M_{j}-1$. Here we note that the special case of $m_{i}=2$ for all $i\geq 1$ we get the sequence partitions into parts of the form $2^{n}-1$, $n\in\mathbb{N}$. They are called $s$-partitions and were introduced in \cite{Bhatt} by Bhatt, who applied them to some problem coming from computer science. The sequence of $s$-partitions was also studied in \cite{Goh, BZ2}.

The following relation is true.

\begin{cor}
Let $k\in\mathbb{N}$. Then for every $n\geq \frac{km_{1}}{m_{1}-1}$ we have
\begin{align*}
a_{M}(n-k,n)=s_{M'}(n).
\end{align*}
\end{cor}
\begin{proof}
By Theorem \ref{Thma(n-k,n)} it is enough to show that the equality
\begin{align*}
s_{M'}(n)=\sum_{k/m_{1}\leq j \leq k/(m_{1}-1)}a_{M'}(m_{1}j-k,j)
\end{align*}
is true for every sequence $M$. We have
\begin{align*}
S_{M'}(x)= &\ F_{M'}(1/x, x^{m_{1}})=\sum_{k=0}^{\infty}\left(\sum_{j=0}^{k}a_{M'}(j,n)x^{-j}\right)x^{m_{1}k} \\
= &\ \sum_{k=0}^{\infty}\left(\sum_{m_{1}\beta -\alpha=k}a_{M'}(\alpha,\beta)\right)x^{k} \\
= &\ \sum_{k=0}^{\infty}\left(\sum_{\beta=0}^{\infty}a_{M'}(m_{1}\beta -k,\beta)\right)x^{k} \\
= &\ \sum_{k=0}^{\infty}\left(\sum_{k/m_{1}\leq \beta \leq k/(m_{1}-1)}a_{M'}(m_{1}\beta -k,\beta)\right)x^{k},
\end{align*}
because $a_{M'}(m_{1}\beta-k,\beta)=0$ if $m_{1}\beta-k<0$ or $m_{1}\beta -k>\beta$, that is, if $\beta <k/m_{1}$ or $\beta >k(m_{1}-1)$. The proof is finished.
\end{proof}

\section*{Acknowledgements}

The author is very grateful to Maciej Ulas, Piotr Miska and Krystian Gajdzica for helpful discussions and for many suggestions.

This paper was a part of the author’s PhD dissertation.

The research of the author was partially supported by the grant of the National Science Centre (NCN), Poland, no. UMO-2019/34/E/ST1/00094.

%On behalf of all authors, the corresponding author states that there is no conflict of interest.

\end{document}